\documentclass[oneside,11pt]{amsart}
\usepackage{amsmath}
\usepackage{amsthm}

\usepackage[T1]{fontenc}

\usepackage[small]{caption}

\usepackage{amsfonts}
\usepackage{amssymb}
\usepackage{amsxtra}

\usepackage[all]{xy}
\SelectTips{cm}{12}

\newtheorem{thm}{Theorem}[section]

\newtheorem*{PropLCZBoundary}{Proposition~\ref{prop:LCZBoundary}}
\newtheorem{prop}[thm]{Proposition}

\newtheorem{cor}[thm]{Corollary}
\newtheorem{conj}[thm]{Conjecture}
\theoremstyle{definition}
\newtheorem{defn}[thm]{Definition}
\newtheorem{rem}[thm]{Remark}

\newtheorem{exmp}[thm]{Example}

\newcommand{\boundary}{\partial}

\newcommand{\bigpresentation}[2]{ \bigl\langle \, {#1} \bigm| {#2} \,
                                  \bigr\rangle }

\newcommand{\bigset}[2]{ \bigl\{ \, {#1} \bigm| {#2} \, \bigr\} }
\renewcommand{\emptyset}{\varnothing}
\renewcommand{\setminus}{-}

\newcommand{\field}[1]{\mathbb{#1}}
\newcommand{\Z}{\field{Z}}
\newcommand{\R}{\field{R}}

\DeclareMathOperator{\CAT}{CAT}

\usepackage{combelow}

\usepackage{graphics}
\usepackage{color}
\usepackage{pinlabel}

\usepackage{ifthen}

\newcommand{\showcomments}{yes}

\newsavebox{\commentbox}
{\ifthenelse{\equal{\showcomments}{yes}}%
{\footnotemark
    \begin{lrbox}{\commentbox}
    \begin{minipage}[t]{1.25in}\raggedright\sffamily\upshape\tiny
    \footnotemark[\arabic{footnote}]}
{\begin{lrbox}{\commentbox}}}%
{\ifthenelse{\equal{\showcomments}{yes}}%
{\end{minipage}\end{lrbox}\marginpar{\usebox{\commentbox}}}
{\end{lrbox}}}

\begin{document}

\title[An introduction to semistability]{%
An introduction to semistability in geometric group theory}

\author[G.C.~Hruska]{G.~Christopher Hruska}
\address{Department of Mathematical Sciences\\
         University of Wisconsin--Milwaukee\\
         PO Box 413\\
         Milwaukee, WI 53211\\
	 USA}
\email{chruska@uwm.edu}

\author[K.~Ruane]{Kim Ruane}
\address{Department of Mathematics\\
         Tufts University\\
         Medford, MA 02155\\
	 USA}
\email{kim.ruane@tufts.edu}

\begin{abstract}

A finitely presented group is semistable at infinity if all proper rays in the Cayley $2$--complex are properly homotopic.  A long standing open question asks whether all finitely presented groups are semistable at infinity. This article provides a brief introduction to the notion of semistability at infinity in geometric group theory. We discuss techniques for proving semistability at infinity that involve either the topology of the boundary or the existence of certain hierarchies of splittings of the group.

As an illustration of the second technique, we apply a combination theorem of Mihalik--Tschantz to prove semistability at infinity for groups that are hyperbolic relative to polycyclic subgroups using work of Mihalik--Swenson and Louder--Touikan.
\end{abstract}



\date{\today}

\maketitle

\section{Introduction}
\label{sec:Introduction}

In geometric group theory it is often useful to characterize the shape of a noncompact space ``near infinity.''
One-ended spaces are, in a sense, connected at infinity.
The notion of semistability at infinity is analogous to being path connected at infinity.  If $X$ is semistable at infinity, then it has a canonical fundamental group at infinity. Informally the fundamental group at infinity of a space $X$ detects nontrivial loops in the part of $X$ that is ``near infinity.''

Let $X$ be a one-ended, locally finite, simply connected CW complex. Formally, a space $X$ is \emph{semistable at infinity} if all proper rays in $X$ are properly homotopic. A \emph{proper ray} in $X$ is a map $[0,\infty) \to X$ that is proper in the sense that the preimage of any compact set is compact.  Two such rays are \emph{properly homotopic} if they are connected by a homotopy that is a proper map.
The ordinary fundamental group of a space $Y$ is defined with respect to a choice of basepoint, and may give different results for basepoints not in the same path component of $Y$. The fundamental group at infinity of $X$ is defined with respect to a choice of proper ray, which plays the role of the basepoint, and may give different results for proper rays that are not properly homotopic.
See Section~\ref{sec:FundamentalGroup} for more details. 

More generally one could study locally finite, simply connected CW complexes $X$ that are not one ended by restricting the above discussion to a single end of $X$.
One says that a multi-ended space is semistable at infinity if any two proper rays converging to the same end are properly homotopic.

A finitely presented group $G$ is \emph{semistable at infinity} if the universal cover $\tilde{X}$ of a presentation $2$--complex $X$ for $G$ is semistable at infinity.
Many significant groups of interest are known to be semistable at infinity, including virtually polycyclic groups \cite{Mihalik83}, word hyperbolic groups \cite{BestvinaMess91,Swarup96,Levitt98,Bowditch99Treelike}, finitely generated Coxeter groups and Artin groups \cite{Mihalik96_ArtinCoxeter}, Baumslag--Solitar groups \cite{Mihalik83}, almost all mapping class groups \cite{Harer86}, $\text{Out}(F_n)$ for $n\ge 2$ \cite{BestvinaFeighn00_OutFn}, $\text{GL}_n(\Z)$ for $n\ge 2$ \cite{Vogtmann95}, and many others.
In fact, there are no known examples of a finitely presented group that is not semistable at infinity.
The general question of whether all finitely presented groups are semistable at infinity has been promoted by Geoghegan and Mihalik for over 35 years, and appears to be quite difficult in general (see for example \cite{Mihalik83}).

The main goal of this paper is to discuss the general ideas behind semistability at infinity and some historical background, to give several natural consequences of semistability in quite different settings, and to exhibit several methods for recognizing when a group is semistable at infinity.

In Section~\ref{sec:FundamentalGroup}, we introduce the fundamental group at infinity.
Section~\ref{sec:Significance} discusses inverse sequences, aspherical manifolds, and group cohomology.
In Section~\ref{sec:Examples}, we focus on interpretations of semistability in the setting of $\CAT(0)$ spaces.
Section~\ref{sec:hierarchies} focuses on graph of groups splittings, hierarchies of splittings, and an application to relatively hyperbolic groups.

The application uses results of Mihalik--Tschantz \cite{MihalikTschantz92}, Louder--Touikan \cite{LouderTouikan17} and Mihalik--Swenson \cite{MihalikSwenson} to prove the following.\footnote{Since the circulation of the initial version of this article, Haulmark--Mihalik have shown the following: If $(G,\mathbb{P})$ is relatively hyperbolic, $G$ is finitely presented and each $P\in \mathbb{P}$ is semistable at infinity then $G$ is semistable at infinity \cite{HaulmarkMihalik}.  Although this result generalizes Theorem~\ref{thm:MainThm}, the proof is quite different from the one given in this article.}
For technical reasons the results of Louder--Touikan only apply under the assumption that the group has no non-central element of order two.

\begin{thm}
\label{thm:MainThm}
Let $(G,\mathbb{P})$ be relatively hyperbolic with no non-central element of order two.
Assume each peripheral subgroup $P \in \mathbb{P}$ is slender and coherent and all subgroups of $P$ are semistable.
Then $G$ is semistable.
\end{thm}

A group is \emph{slender} if all of its subgroups are finitely generated, and it is \emph{coherent} if all finitely generated subgroups are finitely presented. The only slender, coherent groups that the authors are aware of are the virtually polycyclic groups.
Since virtually polycyclic groups are known to be semistable by \cite{Mihalik83} (and all of their subgroups are again virtually polycyclic) the following corollary is immediate.

\begin{cor}
\label{cor:Polycyclic}
Let $(G,\mathbb{P})$ be relatively hyperbolic with no non-central element of order two.
If each $P \in \mathbb{P}$ is virtually polycyclic, then $G$ is semistable.
\end{cor}

\subsection{Acknowledgements}
The authors are grateful to Genevieve Walsh and Lars Louder for helpful conversations about hierarchies.
We also thank the referee for useful feedback.
This work was partially supported by a grant from the Simons Foundation (\#318815 to G. Christopher Hruska).

\section{The fundamental pro-group at infinity}
\label{sec:FundamentalGroup}

This section contains a sketch of the definition of the fundamental pro-group at infinity of a one-ended, locally finite, simply connected CW complex.
The fundamental pro-group is defined in terms of inverse sequences of groups, considered up to a certain equivalence relation.  The equivalence classes are called pro-groups.
In this section we focus on a definition that involves a single inverse sequence, but the construction depends on certain arbitrary choices.  We highlight the significance of proper rays and proper homotopies in this definition.  We return to the well-defined issue in the next section, where we discuss the equivalence relation used in the definition of a pro-group.

For more a complete treatment of fundamental pro-groups at infinity, we refer the reader to the extensive survey article by Guilbault \cite{Guilbault16} and the thorough textbook by Geoghegan \cite{Geoghegan08}.

\begin{exmp}
We begin with the example of $\R^n$, which is one-ended whenever $n\ge 2$.
The complement of a large metric ball centered at the origin in the plane $\R^2$ is an open annulus, which has infinite cyclic fundamental group.  When one increases the radius, the natural inclusions of annuli induce isomorphisms on the fundamental group.  Informally, the loops in the plane that lie near infinity (considered modulo homotopies that also lie near infinity) form an infinite cyclic group.  (We are temporarily ignoring basepoint issues here.)
Formally, the fundamental pro-group at infinity is represented by an infinite sequence of infinite cyclic groups connected by isomorphisms.
In this case, one says that the fundamental pro-group is ``stably isomorphic to $\Z$.''

On the other hand, the analogous construction in $\R^n$ for $n\ge 3$ leads to an infinite sequence of trivial groups. In this higher dimensional case, we say that the space is ``stably trivial'' or ``simply connected at infinity.''
\end{exmp}

\begin{rem}
The fundamental pro-group at infinity involves the portion of a space that lies ``near infinity,'' but it does not directly involve any compactification by a boundary at infinity.  The actual definition of the fundamental pro-group at infinity generalizes the above examples.
\end{rem}

Let $X$ be a one-ended, locally finite, simply connected CW complex with a fixed basepoint $x_0$.  Consider an exhaustion of $X$ by compact sets $\{K_n\}$ such that the closure of $K_n$ is contained in the interior of $K_{n+1}$ for each $n=1,2,3,\dots$.
Let $U_n$ be the complement $X - K_n$.
We naturally obtain a nested sequence of open subspaces with inclusions
\[
   U_1 \longleftarrow U_2 \longleftarrow 
   \cdots \longleftarrow U_k \longleftarrow\cdots
\]
The fundamental pro-group at infinity of $X$ is represented by an inverse sequence of groups, where the $n$th group in the sequence is the fundamental group of $U_n$.
A basic difficulty becomes immediately apparent, since fundamental groups require a choice of basepoint.  However the intersection $\bigcap U_i$ is empty, so there does not exist a single point that could serve as the common basepoint of every $U_i$.

We can fix this problem by choosing a basepoint $x_i$ in the unbounded component of $U_i$ for each $i$.  In order to define a map of groups induced by the inclusion $U_i \to U_{i-1}$ we also need to choose a path from $x_i$ to $x_{i-1}$ that can be used for change of basepoint.  Ideally we would choose this connecting path within the open set $U_{i-1}$.
Such choices allow one to define an inverse sequence of groups
\[
   \pi_1(U_1,x_1) \longleftarrow
   \pi_1(U_2,x_2) \longleftarrow \cdots
   \longleftarrow
   \pi_1(U_n,x_n) \longleftarrow \cdots
\]
where the bonding maps are induced by change of basepoint along the chosen paths.  We discuss inverse sequences of groups more carefully in Section~\ref{sec:Significance}.
The path $r \colon [0,\infty) \to X$ formed by concatenating the chosen paths is a proper ray in $X$.
This inverse sequence of groups represents the \emph{fundamental pro-group at infinity} of $X$ with respect to the \emph{base ray} $r$.

In general it is more natural to start with an arbitrary proper ray $r$, and use it to define a sequence of basepoints and connecting paths in the following manner.
Let $r$ be a proper ray with $r(0)=x_0$.
For each $i$ choose a positive number $t_i$ such that the image $r[t_i,\infty)$ lies inside $U_i$, and choose $x_i = r(s_i)$ for some $s_i >t_i$.
For convenience we let $s_0=0$, and we choose $s_i > s_{i-1}$ for each $i$.

We have been a bit informal with the details above, since the construction described here seems to depend on various arbitrary choices such as the initial basepoint $x_0$, the proper ray $r$, and the sequence of basepoints $x_i \in U_i$.
A more careful treatment reveals that many of these choices ultimately do not matter and lead to inverse sequences that are equivalent in a natural sense.
However the choice of proper ray is highly significant.
It turns out that the inverse sequences of fundamental groups with respect to different proper rays $r$ and $r'$ give equivalent inverse sequences whenever the rays $r$ and $r'$ are joined by a proper homotopy.  

Recall that a one-ended space $X$ is semistable at infinity if all proper rays in $X$ are properly homotopic.
In this case, if $X$ is semistable at infinity, it has a canonical fundamental pro-group at infinity.
As mentioned in the introduction, a multi-ended space is semistable at infinity if any two proper rays converging to the same end are properly homotopic.
This property implies that each end of $X$ has a well-defined fundamental pro-group at infinity.

\section{Significance of semistability}
\label{sec:Significance}

This section focuses on the fundamental group at infinity, as an inverse sequence of groups.
We examine several basic algebraic flavors of inverse limits, and their topological interpretations for the fundamental group at infinity (see Theorem~\ref{thm:GeometricFlavors}). We discuss a characterization of semistability in terms of inverse sequences whose bonding maps are all surjective, see Theorem~\ref{thm:MittagLeffler}.

One of the most well-known examples of a non-semistable space is Whitehead's contractible open $3$--manifold, see below. Although it is not semistable at infinity, the Whitehead manifold also does not admit a proper, cocompact group action (see Theorem~\ref{thm:Wright}).
This naturally leads to semistability conjectures for closed, aspherical manifolds and for finitely presented groups (Theorems \ref{conj:SemistableManifold} and~\ref{conj:SemistableFP}).
We also mention a consequence of semistability regarding the structure of $H^2(G; \Z G)$.

\begin{defn}
An \emph{inverse sequence} of groups $(G_i, p_i)$ is a sequence of groups and homomorphisms (called \emph{bonding maps})
\[
   G_0 \xleftarrow[\qquad]{p_1} G_1 \xleftarrow[\qquad]{p_2} G_2 \xleftarrow[\qquad]{p_3} \cdots
\]
The \emph{inverse limit} of an inverse sequence is the group
\[
   \varprojlim (G_i,p_i) = 
   \bigset{(g_0,g_1,g_2,\dots)\in \prod G_i}{p_i(g_i) = g_{i-1}}
\]
\end{defn}

The fundamental group at infinity of a space $X$ with respect to a base ray $r$ is an inverse sequence of group that depends on a choice of exhaustion $\{K_m\}$ of $X$ by compact sets.
If one were to choose a different exhaustion $\{L_n\}$, the two inverse sequences can be related in a natural way.

First note that any compact set of $X$ must be contained in $K_m$ for some $m$ and likewise, contained in $L_n$ for some $n$. 
Therefore there exist subsequences $\{K_{m_i}\}$ and $\{L_{n_i}\}$ of the given exhaustions that interleave with each other in the sense that $K_{m_i} \subset L_{n_i}$ and $L_{n_i} \subset K_{m_{i+1}}$ for all $i$.
Taking complements, as in the definition of fundamental group at infinity, we see that the two inverse sequences obtained from the two exhaustions are equivalent in the following sense.

\begin{defn}
Two inverse sequences $(G_i,p_i)$ and $(H_j,q_j)$ are \emph{equivalent}
if, after passing to subsequences, there exists a commutative ladder diagram
\[
\xymatrix@C-18pt{
   G_{i_0} & & G_{i_1} \ar[ll] \ar[dl] & & G_{i_2} \ar[ll] \ar[dl] & \cdots \ar[l]  \\
   & H_{j_0} \ar[ul] & & H_{j_1} \ar[ul] \ar[ll] & & H_{j_2} \ar[ul] \ar[ll] & \cdots \ar[l]
}
\]
\end{defn}

\begin{exmp}[A cautionary example]
\label{exmp:Cautionary}
Consider the inverse sequence
\[
  \Z \xleftarrow[\qquad]{2}
  \Z \xleftarrow[\qquad]{2}
  \Z \xleftarrow[\qquad]{2}
  \cdots
\]
in which each homomorphism is given by multiplication by two.
Then the inverse limit is equal to the trivial group.
However the inverse sequence itself is not equivalent to the trivial inverse system, in which each group $G_i$ is the trivial group.
This inverse sequence arises as the fundamental group at infinity of a dyadic solenoid inverse mapping telescope (see Geoghegan \cite[Ex.~11.4.15]{Geoghegan08}).  See also Marde\v{s}i\'{c}--Segal \cite[Ex.~II.3.1]{MardesicSegal82},  for an alternative interpretation.
\end{exmp}

Because of the previous example, 
when studying the fundamental pro-group at infinity, one needs to keep track of the entire inverse sequence of groups rather than the single group obtained by taking the inverse limit.
In this example, the information contained in the inverse sequence of groups is lost when passing to the inverse limit.

If an inverse sequence $(G_i,p_i)$ is equivalent to a sequence in which all maps are monomorphisms then the inverse sequence is said to be \emph{pro-monomorphic}.
Similarly, we define the inverse sequence to be \emph{pro-epimorphic} if it is equivalent to a system whose maps are all epimorphisms.
The pro-epimorphic property is more commonly known as the \emph{Mittag-Leffler} property, and thus we use that terminology here.
This property appears implicitly in the classical proof of the Mittag-Leffler theorem of complex analysis (see \cite[\S II.3.5]{BourbakiTG}).
An inverse sequence is \emph{pro-trivial} if it is equivalent to a system in which each group $G_i$ is the trivial group. Thus in Example~\ref{exmp:Cautionary}, the inverse limit is trivial but the inverse sequence is not pro-trivial.   

When the inverse sequence in question is the fundamental pro-group at infinity of a space, the properties of inverse sequences defined above have simple geometric interpretations. We refer the reader to \cite[\S 3.4.8]{Guilbault14} for details.

\begin{thm}
\label{thm:GeometricFlavors}
The fundamental group at infinity of a one-ended space $X$ with respect to a base ray $r$ is
\begin{enumerate}
    \item pro-monomorphic if and only if
    there exists a compact set $C \subset X$ such that for every compact set $D$ containing $C$ there exists a compact set $E$ such that every loop in $X \setminus E$ that contracts in $X \setminus C$ contracts in $X \setminus D$,
    \item Mittag-Leffler if and only if
    for any compact set $C \subset X$ there exists a larger compact set $D$ such that for every still larger compact set $E$, each pointed loop in $X \setminus D$ based on $r$ can be homotoped into $X \setminus E$ via a homotopy in $X\setminus C$ that slides the base point along $r$, and
    \item pro-trivial if and only if for every compact set $C\subset X$, there exists a larger compact set $D$, such that every loop in $X\setminus D$ contracts in $X\setminus C$. 
\end{enumerate}
\end{thm}

By parts (1) and (3) of the above theorem, the properties of the fundamental pro-group at infinity being either pro-monomorphic or pro-trivial do not depend on the choice of base ray.
We now examine the Mittag-Leffler property in more detail and discuss its relation to semistability.

One can make precise the informal statement that no information is lost when passing to the limit of a Mittag-Leffler inverse sequence.
In fact, if one is given an exact sequence in which the terms are Mittag-Leffler inverse sequences, then the inverse limit is again an exact sequence. (But inverse limits in general need not preserve exactness.)
A simpler, related phenomenon is that a Mittag-Leffler inverse sequence with a trivial inverse limit must be pro-trivial.
From the point of view of homological algebra, we can interpret the above as evidence for the ``tameness'' of  Mittag-Leffler inverse sequences.
See Marde\v{s}i\'{c}--Segal \cite[\S II.6.2]{MardesicSegal82} for a more detailed discussion of this viewpoint.

The Mittag-Leffler property is closely related to semistability at infinity by the following result, which also implies that Mittag-Leffler does not depend on the choice of base ray.

\begin{thm}[\cite{Geoghegan08}, Prop.~16.1.2]
\label{thm:MittagLeffler}
Let $X$ be a one-ended, locally finite, simply connected CW complex, and let $r$ be a proper ray in $X$. Then $X$ is semistable at infinity if and only if the fundamental group at infinity with respect to the base ray $r$ is a Mittag-Leffler inverse sequence.
\end{thm}

By applying Theorem~\ref{thm:MittagLeffler}, one can show that the Whitehead $3$--manifold is not semistable at infinity.
The Whitehead $3$-manifold was the first known example of a contractible open $3$--manifold not homeomorphic to $\R^3$.
The following example sketches the proofs of these two facts. We refer the reader to Geoghegan \cite[\S 16.4]{Geoghegan08} for detailed proofs of the assertions below.

\begin{exmp}[Whitehead $3$--manifold]
Let $L \subset S^3$ be the Whitehead link, a two-component link such that each component is individually unknotted in $S^3$ and each component contracts to a point in the complement of the other---although the contraction requires the loop to pass through itself. 
A regular neighborhood of $L$ consists of two solid tori with boundary tori $T_0$ and $T_1$.
\begin{center}
\labellist\small\hair 2pt
\pinlabel $T_1$ at 112 15
\pinlabel $T_0$ at 5 5
\endlabellist
\includegraphics{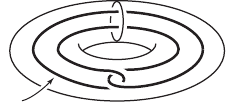}
\end{center}
Since these tori are unknotted, each separates $S^3$ into a pair of solid tori.
Let $M_0$ be the solid torus bounded by $T_0$ such that $M_0$ does not contain $T_1$.
Then there exists a homeomorphism $h\colon S^3 \to S^3$ such that $h(T_0)=T_1$ and $M_0 \subset h(M_0)$.
Consider the nested sequence of solid tori $M_0 \subset M_1 \subset M_2 \subset \cdots$ such that $M_{i+1} = h(M_i)$ for each $i\ge 0$.
The \emph{Whitehead $3$--manifold} is the union $W^3 = \bigcup M_i$.
Since $M_0$ contracts to a point in $M_1$, each $M_i$ contracts to a point in $M_{i+1}$.  Thus by compactness, every map $S^n \to W^3$ is nullhomotopic for each $n$, so that $W^3$ is contractible.

To see that $W^3$ is not homeomorphic to $\R^3$, we examine it from a different point of view.
Consider the compact $3$--manifold $C_i = \overline{M_i \setminus M_{i-1}}$.  Then $W^3$ decomposes as the following sequence of submanifolds (with pairwise disjoint interiors) glued along boundary tori:
\[
   W^3 = M_0 \cup_{T_0} C_1 \cup_{T_1} C_2 \cup_{T_2} C_3 \cup_{T_3} \cdots
\]
The end of $W^3$ has basic neighborhoods of infinity of the form
\[
   U_i = C_i \cup_{T_i} C_{i+1} \cup_{T_{i+1}} C_{i+2} \cup_{T_{i+2}} \cdots
\]
for each $i$.  Each $C_i$ is homeomorphic to $C_1$, the core of the Whitehead link complement, which is a compact $3$--manifold with boundary a pair of incompressible tori. 
In particular, the inclusions of $T_i$ into $C_i$ and $C_{i+1}$ induce injective maps on $\pi_1$ that are not surjective, which implies that $\pi_1(U_i)$ splits as a nontrivial amalgamated product, as follows.
If we let $H_i = \pi_1(C_i) \approx \pi_1(C_1)$, then
\[
   \pi_1(U_i) = H_i *_{\Z^2} \bigl( H_{i+1} *_{\Z^2} H_{i+2} *_{\Z^2} \cdots \bigr) = H_i *_{\Z^2} \pi_1(U_{i+1})
\]
In particular, $\pi_1(U_i)$ is a nontrivial group for all $i$.
It follows from the theory of graphs of groups that the map $\pi_1(U_{i+1}) \to \pi_1(U_{i})$ induced by inclusion is injective.
But since $\Z^2 \to H_i$ is not surjective, the map $\pi_1(U_{i+1}) \to \pi_1(U_{i})$ is not surjective (see, for instance, \cite{ScottWall79}).

The fundamental group at infinity is represented by the inverse sequence
\[
   \pi_1(U_1) \leftarrow \pi_1(U_2) \leftarrow \pi_1(U_3) \leftarrow \cdots
\]
in which each map is injective.  Therefore the fundamental group at infinity of $W^3$ is pro-monomorphic.
Observe that an inverse sequence of nontrivial groups in which all bonding maps are injective cannot be pro-isomorphic to a nontrivial inverse sequence in which all bonding maps are surjective.
In particular, such an inverse sequence is not pro-trivial.
Thus we conclude that $W^3$ is not homeomorphic to $\R^3$.

The argument above also establishes that the fundamental group at infinity does not have the Mittag-Leffler property.
By Theorem~\ref{thm:MittagLeffler}, we conclude that the contractible open $3$--manifold $W^3$ is not semistable at infinity.
\end{exmp}

The existence of the Whitehead manifold suggests the following question: is the Whitehead manifold the universal cover of a closed aspherical $3$--manifold?
More generally, which contractible manifolds arise as universal covers of closed aspherical manifolds?
F.E.A. Johnson conjectured that $\R^n$ is the only contractible manifold arising as such a universal cover (see \cite{Johnson71_AsphericalManifold}).
However, Davis proved that Johnson's conjecture is false by constructing, for each $n\ge 4$, a closed aspherical $n$--manifold whose universal cover is not homeomorphic to $\R^n$ \cite{Davis83}.

The original question regarding the Whitehead $3$--manifold was resolved by Myers \cite{Myers88}: the Whitehead $3$--manifold does not admit a covering space action by any nontrivial group. Wright proved the following theorem, which substantially generalizes Meyers' result.

\begin{thm}[\cite{Wright92}]
\label{thm:Wright}
Suppose $M$ is a contractible open $n$--manifold whose fundamental group at infinity is pro-monomorphic.
If $M$ admits a covering space action by any nontrivial group, then $M$ is homeomorphic to $\R^n$.
\end{thm}

In particular, the examples constructed by Davis are not pro-monomorphic.  On the other hand, it follows from work of Ancel--Siebenmann that the Davis examples are Mittag-Leffler (see \cite[Thm.~3.5.5]{Guilbault16} for a proof).
In fact, there are currently no known examples of a closed aspherical manifold whose universal cover is not Mittag-Leffler.
The following is a well-known open conjecture in geometric topology.

\begin{conj}[Semistability for manifolds]
\label{conj:SemistableManifold}
If $M$ is a closed, aspherical manifold, then $\tilde{M}$ is semistable at infinity.
\end{conj}

Furthermore, there are no known counterexamples to the following much more general conjecture due to Geoghegan and Mihalik mentioned in the introduction.

\begin{conj}[Semistability for finitely presented groups]
\label{conj:SemistableFP}
Let $X$ be any finite, connected $2$--dimensional CW complex.
Then the universal cover $\tilde{X}$ is semistable at infinity.
\end{conj}

If a finitely presented group $G$ is realized as the fundamental group of two different finite, connected $2$--complexes $X$ and $Y$, then 
$\tilde{X}$ is semistable at infinity if and only if $\tilde{Y}$ is (see \cite[\S 16.5]{Geoghegan08}). Thus semistability at infinity is an intrinsic property of the finitely presented group $G$.
More generally, semistability at infinity is a quasi-isometry invariant of finitely presented groups (see \cite[\S 18.2]{Geoghegan08}).

Semistability of finitely presented groups has implications for the algebraic structure of $H^2(G;\Z G)$.  The cohomology group $H^n(G;\Z G)$ is a coarse invariant, \emph{i.e.}, depending only on the quasi-isometry type of $G$, which encodes asymptotic information about the homology of $G$ in dimensions less than $n$
(see \cite[\S 5.1]{Roe_coarse} or \cite[\S 18.2]{Geoghegan08})
.  For example, $H^1(G;\Z G)$ detects the number of ends of $G$.
If $G$ is finitely presented, the abelian group $H^2(G;\Z G)$ is torsion free (see \cite[\S 13.7]{Geoghegan08}).  The following theorem due to Farrell gives more information.

\begin{thm}[\cite{Farrell74}, Cor.~5.2]
If $G$ is a finitely presented group that contains an infinite order element, then $H^2(G; \Z G)$ is either trivial, infinite cyclic, or not finitely generated.
\end{thm}

One might wonder whether $H^2(G;\Z G)$ is always free abelian.  Indeed, this question has been attributed to Hopf.  If $G$ is a Poincar\'e\ Duality group, such as the fundamental group of a closed aspherical manifold, then $H^2(G;\Z G)$ is either trivial or infinite cyclic.
But for arbitrary finitely presented groups, the answer to Hopf's question is unknown.  The following result due to Geoghegan--Mihalik relates Hopf's question to the semistability at infinity of $G$.

\begin{thm}[\cite{Geoghegan08}, Thm.~16.5.1]
If G is semistable at infinity then $H^2(G;\Z G)$ is free abelian.
\end{thm}

\section{Semistability in the $\CAT(0)$ setting}
\label{sec:Examples}

In this section, we discuss various interpretations of semistability in the $\CAT(0)$ setting.  
So far we have only discussed semistability in the settings of manifolds and locally finite CW complexes.
In order to include $\CAT(0)$ spaces in our discussion, we need to extend the class of spaces under consideration.

A metrizable space $X$ is an \emph{absolute neighborhood retract} or \emph{ANR} if whenever $X$ is embedded as a closed subspace of a metrizable space $Y$, there exists a neighborhood of $X$ in $Y$ that retracts onto $X$.
An \emph{absolute retract} or \emph{AR} is a contractible ANR.
A metrizable space $X$ is an ANR if $X$ is finite dimensional and locally contractible or if every point of $X$ has a neighborhood that is an ANR (see Hu \cite{Hu65}).
Consequently the family of ANRs includes all manifolds and all locally finite CW complexes.
Ontaneda has shown that every proper $\CAT(0)$ space is an absolute retract in \cite{Ontaneda05}.

Above we defined a group to be semistable at infinity if the universal cover of a presentation $2$--complex is semistable at infinity.  Semistability of a group $G$ does not depend on the choice of presentation $2$--complex.  More generally if $G$ acts properly  and cocompactly by homeomorphisms on one-ended, simply connected, locally compact ANRs $X$ and $Y$, then there exists a proper homotopy equivalence between $X$ and $Y$.  In particular, if $X$ is semistable at infinity, then so is $Y$. See Guilbault--Moran \cite{GuilbaultMoran_ANR} for details.
Therefore in order to determine whether a one-ended, finitely presented group $G$ is semistable at infinity, one can examine any simply connected, locally compact ANR on which $G$ acts properly and cocompactly by homeomorphisms.

Recall that a $\CAT(0)$ space $X$ is semistable at infinity if every pair of proper rays in $X$ are properly homotopic.  This definition uses only the topology of $X$ but not the $\CAT(0)$ metric.
The following result of Geoghegan--Swenson incorporates the $\CAT(0)$ geometry by showing that it suffices to consider pairs of geodesic rays.

\begin{thm}[\cite{GeogheganSwenson_Semistable}]
\label{thm:GeodesicRays}
A proper one-ended $\CAT(0)$ space $X$ is semistable at infinity if and only if each pair of geodesic rays in $X$ is properly homotopic.
\end{thm}

A $\CAT(0)$ space typically contains many proper rays that are not within a bounded distance of a geodesic ray.
For example one could consider a logarithmic spiral in the Euclidean plane.
The existence of such ``nongeodesic'' rays makes the proof of Theorem~\ref{thm:GeodesicRays} nontrivial.
Indeed, a priori it is not obvious that every proper ray is properly homotopic to a geodesic ray.
The proof of this fact depends on a nontrivial shape-theoretic result of Krasinkiewicz--Minc \cite{KrasinkiewiczMinc79}.

A proper $\CAT(0)$ space $X$ has a natural ``boundary at infinity,'' the \emph{visual boundary}, which encodes the geometry of geodesic rays in $X$.
Two geodesic rays $c,c'\colon [0,\infty) \to X$ are \emph{equivalent} if the function $d \bigl( c(t),c'(t) \bigr)$ is bounded above as $t\to\infty$.
The \emph{visual boundary} $\boundary X$ is the set of equivalence classes.
For any fixed basepoint $x_0$, each equivalence class is represented by a unique geodesic ray emanating from $x_0$.
Two rays are ``close'' if their representatives $c,c'$ based at $x_0$ have $d \bigl( c(T),c'(T) \bigr)$ ``small'' for a large value of $T$ (see \cite{Ballmann95} or \cite{BH99} for a more precise study of this topology).

By Theorem~\ref{thm:GeodesicRays}, establishing semistability in a $\CAT(0)$ space only requires one to connect an arbitrary pair of \emph{geodesic} rays $c_0,c_1$ by a proper homotopy.  However, the intermediate proper rays $c_t$ in this homotopy need not be geodesic.
When the rays $c_t$ in such a homotopy are geodesic, then 
there is a path in the boundary between $c_0$ and $c_1$. Conversely, if the boundary is path-connected, 
then the above criterion is easily satisfied.  

Since connected and locally connected implies path-connected, we have the following corollary. 

\begin{cor}
\label{cor:PathConnected}
A proper one-ended $\CAT(0)$ space $X$ is semistable at infinity if its visual boundary $\boundary X$ is path connected.
In particular, if $\boundary X$ is locally connected then $X$ is semistable at infinity.
\end{cor}

We will use the boundary to establish various examples of semistable $\CAT(0)$ groups.
In particular, we demonstrate examples with locally connected boundary, with non--locally connected but path connected boundary, and also with non-path connected boundary.

It is unknown whether all word hyperbolic groups are $\CAT(0)$.
However Bestvina--Mess have shown that the 
Rips complex may be used in place of a $\CAT(0)$ space to obtain the same conclusion as Corollary~\ref{cor:PathConnected} for all word hyperbolic groups (see \cite{BestvinaMess91}).
One concludes that one-ended hyperbolic groups are all semistable at infinity, since the Gromov boundary is always locally connected \cite{Swarup96,Levitt98,Bowditch99Treelike}.

Many $\CAT(0)$ groups do not admit locally connected visual boundaries (see, for example, \cite{MihalikRuane99,MihalikRuane01}).  In that case, the boundary may still be path connected.
For example, a theorem of Ben-Zvi states that the boundary of a one-ended $\CAT(0)$ group with isolated flats is always path connected \cite{Benzvi_PathCon}.  Thus all such groups are semistable at infinity.

Another example of this type is the group $F_2 \times \Z$, which acts on the product of a tree $T$ and a line $\R$.
A product of any two $\CAT(0)$ spaces $X\times Y$ is again a $\CAT(0)$ space (using the product metric) whose boundary is homeomorphic to the join $\boundary X * \boundary Y$ of the boundaries of the factors. In particular, the boundary of $T \times \R$ is homeomorphic to the suspension of $\boundary T$, in other words, it is the suspension of a Cantor set, which is not locally connected.
(Indeed, if $A$ is not locally connected, then the suspension of $A$ is also not locally connected, for any topological space $A$.)

More generally, by a theorem of Bowers--Ruane, if $F_2 \times \Z$ acts geometrically on any other $\CAT(0)$ space $Y$, the visual boundary $\boundary Y$ is always homeomorphic to the suspension of a Cantor set \cite{BowersRuane96}, so it is path connected but not locally connected.

The examples discussed so far all have path connected visual boundaries, and are semistable at infinity by Corollary~\ref{cor:PathConnected}. It follows from Corollary~\ref{cor:PathConnected}, that a $\CAT(0)$ group $G$ is semistable at infinity if $G$ acts properly, cocompactly, and isometrically on at least one $\CAT(0)$ space which has path connected visual boundary. 
However it is possible for a one-ended $\CAT(0)$ group to be semistable at infinity even when it does not admit a path connected visual boundary.

A well-known example to consider in this context is the following example due to Croke--Kleiner \cite{CrokeKleiner00}.
The \emph{Croke--Kleiner group} is the group $G$ defined by the presentation $\bigpresentation{a,b,c,d}{[a,b], [b,c], [c,d]}$.
The presentation $2$--complex $X$ of $G$ naturally has the structure of a nonpositively curved square complex, whose universal cover $\tilde{X}$ is a one-ended $\CAT(0)$ space.
The visual boundary $\boundary \tilde{X}$ is connected but not path connected.  (This result is implicit in the general theory of \cite{CrokeKleiner02} and was explicitly verified by Conner--Mihalik--Tschantz \cite{ConnerMihalikTschantz}.)
Although the techniques of \cite{CrokeKleiner02} apply to any $\CAT(0)$ space $Y$ on which $G$ acts geometrically, it is unclear whether the boundary $\boundary Y$ is not path connected in general.

However the Croke--Kleiner group  $G$ is semistable at infinity by the following combination theorem of Mihalik--Tschantz.

\begin{thm}[\cite{MihalikTschantz92}, Mihalik--Tschantz Combination Theorem]
\label{thm:MihalikTschantz}
Let $H$ split as a finite graph of groups $\mathcal{G}(H)$.
Assume all vertex groups are finitely presented and all edge groups are finitely generated.
If all vertex groups are semistable at infinity, then so is $H$.
\end{thm}

Indeed, the Croke--Kleiner group splits as an amalgam
\[
   G = \langle{a,b,c}\rangle *_{\langle{b,c}\rangle}
   \langle b,c,d \rangle
\]
of two groups each isomorphic to $F_2\times \Z$.
As explained above $F_2 \times \Z$ is semistable at infinity, so the Croke--Kleiner group is as well.

For groups $G$ that act geometrically, \emph{i.e.}, properly, cocompactly, and isometrically on a $\CAT(0)$ space $X$, semistability has an equivalent shape-theoretic characterization in terms of the topology of the visual boundary $\boundary X$. Intuitively,  a $\CAT(0)$ space is semistable at infinity if it is ``locally connected at infinity'' in a sense that is made precise in slightly different ways in the following two theorems.

\begin{thm}[\cite{Krasinkiewicz77}, \cite{DydakSegal78} Thm.~7.2.3]
\label{thm:LocallyConnected}
Suppose $G$ is one ended and acts geometrically on a $\CAT(0)$ space $X$.
Then $G$ is semistable at infinity if and only if $\boundary X$ is shape equivalent to a locally connected continuum.
\end{thm}

We will not need to use the technical definition of shape equivalence in this section, so we defer the definition to Appendix~\ref{sec:ShapeTheory}.
Shape equivalence of compacta is weaker than homotopy equivalence.  For example the Cantor Hawaiian earring is not locally connected, but it is shape equivalent to the Sierpinski carpet and also to the Hawaiian earring, which are locally connected.
On the other hand, the dyadic solenoid is not shape equivalent to any locally connected space.

Although Theorem~\ref{thm:LocallyConnected} is expressed in the language of shape theory, it has an interpretation in the language of geometric group theory.  This interpretation involves the notion of a $\mathcal{Z}$--boundary of a space, a notion that first arose in infinite dimensional topology.  Bestvina--Mess demonstrate the significance of this notion in geometric group theory by applying it to the study of word hyperbolic groups and $\CAT(0)$ spaces \cite{BestvinaMess91,Bestvina96}.

\begin{defn}[$\mathcal{Z}$--set and $\mathcal{Z}$--boundary]
A closed subspace $A$ of a compact absolute retract $Y$ is a \emph{$\mathcal{Z}$--set} in $Y$ if for each $\epsilon>0$ there exists a map $f\colon Y \to Y\setminus A$ with $d(f,1_Y) < \epsilon$.
If $X$ is a locally compact AR, we say that $A$ is a \emph{$\mathcal{Z}$--boundary} of $X$ if $A$ is a $\mathcal{Z}$--set in a compact AR $Y$ such that $Y \setminus A$ is homeomorphic to $X$.

For example, if $X$ is a locally compact $\CAT(0)$ space then the visual boundary $\boundary X$ is a $\mathcal{Z}$--boundary of $X$.
Indeed, $Y = X \cup \boundary X$ has a natural visual topology such that $Y$ is a compact metrizable AR and  $\boundary X$ is a $\mathcal{Z}$--set in $Y$.
\end{defn}

\begin{prop}
\label{prop:LCZBoundary}
Suppose $G$ is one ended and acts geometrically on a $\CAT(0)$ space $X$.
Then $G$ is semistable at infinity if and only if $G$ acts geometrically on a $\CAT(0)$ space $X'$ that admits a locally connected $\mathcal{Z}$--boundary.
\end{prop}

The proof of Proposition~\ref{prop:LCZBoundary} can be found in Appendix~\ref{sec:ShapeTheory}.

\section{Hierarchies and relative hyperbolicity}
\label{sec:hierarchies}

By Mihalik--Tschantz's Combination Theorem (Theorem~\ref{thm:MihalikTschantz}) one can prove that a group is semistable at infinity by showing it splits as a graph of groups with semistable vertex groups.  One can reach a broad family of groups by considering graphs of groups whose vertex groups split again as a graph of groups.  In fact one can iterate this process many times.
Iterated graphs of groups lead to the notion of hierarchies of splittings, which we examine in this section.  In particular, we use this strategy to prove semistability at infinity for a large class of relatively hyperbolic groups using a hierarchy theorem of Louder--Touikan. 

\begin{defn}
A graph of groups splitting $\mathcal{G}(G)$ of a group $G$ is \emph{nontrivial} if $G$ is not contained in a vertex group.
If $\mathcal{A}$ is a family of subgroups of $G$, then a splitting is \emph{over $\mathcal{A}$} if every edge group is a member of the family $\mathcal{A}$.
A splitting $\mathcal{G}(G)$ is \emph{relative} to a family of subgroups $\mathcal{P}$ if each $P \in \mathcal{P}$ is contained in a conjugate of a vertex group of $\mathcal{G}(G)$.
\end{defn}

\begin{defn}
A \emph{hierarchy} $\mathcal{H}(G)$ of a group $G$ is a rooted tree of groups with $G$ at the root such that the descendants of a group $H \in \mathcal{H}(G)$ are the vertex groups of a nontrivial graph of groups decomposition $\mathcal{G}(H)$ of $H$.  A group $H \in \mathcal{H}(G)$ is \emph{terminal} if $H$ has no descendants.
A hierarchy is \emph{slender} if each of its graphs of groups $\mathcal{G}(H)$ is a splitting over slender groups.
A hierarchy is \emph{finite} if the underlying rooted tree is finite.
\end{defn}

An obvious corollary to Theorem~\ref{thm:MihalikTschantz} is the following Hierarchy Combination Theorem.

\begin{cor}
\label{cor:HierarchyCombination}
Suppose $G$ admits a finite hierarchy with all vertex groups finitely presented and all edge groups finitely generated.  If all terminal vertex groups are semistable at infinity then so is $G$.
\end{cor}

The preceding corollary requires knowledge of the finiteness properties of every vertex group in the hierarchy.  In some cases the following more restrictive result may be easier to apply, since it only requires knowledge of the edge groups.

\begin{cor}
\label{cor:SemistableHierarchy}
Suppose $G$ is finitely presented and admits a finite hierarchy with finitely presented edge groups.
If all terminal vertex groups are semistable at infinity, then so is $G$.
\end{cor}

\begin{proof}
If $H$ is finitely presented and splits as a graph of groups $\mathcal{G}(H)$ with finitely presented edge groups then every vertex group of $\mathcal{G}(H)$ is finitely presented (see \cite[Lem.~1.1]{Bowditch99Connectedness}).
The group $G$ is finitely presented by hypothesis.
By induction on the levels of the hierarchy, it follows that each vertex group in the hierarchy is finitely presented.
By Corollary~\ref{cor:HierarchyCombination}, we are done.
\end{proof}

In order to apply Corollary~\ref{cor:SemistableHierarchy}, one needs a family of groups for which finite hierarchies are known to exist and terminate in groups that are semistable at infinity.
In fact, finite hierarchies are known to exist for the classes of one-relator groups, three-manifold groups, limit groups, and hyperbolic cubulated groups (see Wise \cite{WiseHierarchies} for more details on these families of groups and their hierarchies).  A result of Louder--Touikan shows that certain relatively hyperbolic groups admit finite hierarchies, as described below.  We begin with the definition of relative hyperbolicity. 

\begin{defn}[Relatively hyperbolic]
A graph is \emph{fine} if for each positive integer $n$, each edge is contained in only finitely many circuits of length $n$.
If $G$ is a group and $\mathbb{P}$ is a finite collection of infinite subgroups of $G$, 
we say that $(G,\mathbb{P})$ is \emph{relatively hyperbolic} if $G$ acts cocompactly on a fine $\delta$--hyperbolic graph $K$ such that each edge stabilizer is finite and the set of all conjugates of members of $\mathbb{P}$ is equal to the set of all infinite vertex stabilizers.

A subgroup of $G$ is \emph{parabolic} if it has a conjugate that lies in a member of $\mathbb{P}$.
A subgroup is \emph{elementary} if it is either finite, two-ended, or parabolic.
\end{defn}

\begin{exmp}
The simplest examples of fine graphs are trees and locally finite graphs.
If $G$ is word hyperbolic, one sees that $(G,\emptyset)$ is relatively hyperbolic by examining the action on the Cayley graph for any finite generating set.
If $G$ is finitely generated and multi-ended, then Stallings' Theorem states that $G$ acts cocompactly on a simplicial tree $T$ such that every edge stabilizer is finite (see, for example, \cite{ScottWall79}).
Therefore $(G,\mathbb{P})$ is relatively hyperbolic, where $\mathbb{P}$ is a set of representatives of the vertex stabilizers.

By a theorem of Farb \cite{Farb98}, if $M$ is a complete, finite volume manifold with pinched negative sectional curvature, then $\bigl(\pi_1(M),\mathbb{P}\bigr)$ is relatively hyperbolic, where $\mathbb{P}$ is the family of fundamental groups of the cusps of $M$.
Farb proves relative hyperbolicity by constructing a fine hyperbolic graph known as the coned-off Cayley graph.
\end{exmp}

\begin{rem}[Slender $\Longrightarrow$ elementary]
\label{rem:SlenderElementary}
By the Tits Alternative, every nonelementary subgroup of a relatively hyperbolic group contains a copy of the free group of rank two \cite[\S 8.2.F]{Gromov87}.
Since slender groups do not contain $F_2$, all slender subgroups of a relatively hyperbolic group must be elementary.
\end{rem}

A relatively hyperbolic group $(G,\mathbb{P})$ is \emph{atomic} if $G$ is one ended, every peripheral subgroup $P \in \mathbb{P}$ is one ended, and $G$ does not split over a parabolic subgroup. 

\begin{thm}[\cite{MihalikSwenson}]
\label{thm:MihalikSwenson}
Non-elementary atomic relatively hyperbolic groups are semistable at infinity.
\end{thm}

The atomic groups are in a sense the simplest possible type of relatively hyperbolic group.  
A study of hierarchies of relatively hyperbolic groups is necessary to understand the general case.  Recent work of Louder--Touikan \cite{LouderTouikan17} provides such an understanding for splittings over slender groups.  For technical reasons the results of Louder--Touikan only apply under the assumption that the group has no non-central element of order two.

\begin{thm}[\cite{LouderTouikan17}, Cor.~2.7]
\label{thm:LouderTouikan}
Suppose $(G,\mathbb{P})$ is relatively hyperbolic and $G$ is finitely generated with no non-central element of order two.
Then there exists a finite hierarchy $\mathcal{H}(G)$ with the following properties.
Every edge group is slender and elementary in $(G,\mathbb{P})$.
Every terminal vertex group $H$ has a relatively hyperbolic structure $(H,\mathbb{Q})$ such that each $Q \in \mathbb{Q}$ is infinite but not two-ended and $H$ has no nontrivial splittings over slender subgroups relative to $\mathbb{Q}$. 
Furthermore each $Q$ is parabolic in the original peripheral structure $(G,\mathbb{P})$.
\end{thm}

The proof of Theorem~\ref{thm:MainThm} depends on the following elementary result of Bowditch on splittings of relatively hyperbolic groups.

\begin{prop}[\cite{Bowditch01}, Prop.~5.2]
\label{prop:BowditchNoSplitting}
Let $(G,\mathbb{P})$ be relatively hyperbolic.
Suppose each peripheral subgroup $P\in \mathbb{P}$ is one ended.
If $G$ does not split over a parabolic subgroup relative to $\mathbb{P}$,
then $G$ does not admit any splitting over a parabolic subgroup \textup{(}without regard to whether the splitting is relative to $\mathbb{P}$\textup{)}.
\end{prop}

\begin{proof}[Proof of Theorem~\ref{thm:MainThm}]
By Remark~\ref{rem:SlenderElementary}, all slender subgroups of $G$ are elementary. Therefore the hierarchy of $G$ given by Theorem~\ref{thm:LouderTouikan} must involve only elementary edge groups---i.e., finite, two-ended, or parabolic.
By hypothesis, all parabolic subgroups are finitely presented.  Thus all edge groups of the hierarchy are finitely presented.
Since each $P \in \mathbb{P}$ is finitely presented, $G$ is as well, by \cite{Osin06}.

In order to apply Corollary~\ref{cor:SemistableHierarchy} we will show that all terminal vertex groups $H$ are atomic with respect to the peripheral structure $(H,\mathbb{Q})$ given by Theorem~\ref{thm:LouderTouikan}.
We verify that claim below.  The theorem then follows easily from this claim.
Indeed recall that atomic groups are one ended, by definition.
Therefore if a terminal vertex group is elementary, it must be parabolic.  By hypothesis all parabolic subgroups are semistable.  In the non-elementary case, we obtain semistability by applying Theorem~\ref{thm:MihalikSwenson}.

We proceed with the proof of the claim.
The hierarchy terminates with relatively hyperbolic groups $(H,\mathbb{Q})$ such that each $Q \in \mathbb{Q}$ is infinite but not two-ended and each is a subgroup of a $P\in\mathbb{P}$.  By hypothesis all parabolic subgroups of $(G,\mathbb{P})$ are slender and coherent, therefore so are the parabolic subgroups of $(H,\mathbb{Q})$.  
Note that an infinite group $Q$ that is slender and coherent and not two-ended must be one ended. Thus each $Q \in \mathbb{Q}$ is one ended.

Additionally the terminal vertex groups $(H,\mathbb{Q})$ admit no slender splittings relative to $\mathbb{Q}$.
By hypothesis all elementary subgroups are slender, so $H$ admits no elementary splitting relative to $\mathbb{Q}$.
In particular, $H$ does not split over a finite group relative to $\mathbb{Q}$. Since each $Q \in \mathbb{Q}$ is one ended, it follows that $H$ cannot split over a finite group. Indeed, any nontrivial splitting of $H$ over a finite group would induce such a nontrivial splitting of some $Q$, which is impossible.  Thus $H$ is one ended.

Since $(H,\mathbb{Q})$ has no elementary splitting relative to $\mathbb{Q}$, we see that $(H,\mathbb{Q})$ has no nontrivial splitting over a parabolic subgroup relative to $\mathbb{Q}$.
Since each $Q$ is one ended, it follows from Proposition~\ref{prop:BowditchNoSplitting} that $H$ does not have any splitting over a parabolic subgroup.
Therefore the hierarchy terminates with atomic relatively hyperbolic groups, establishing the claim.
\end{proof}

\appendix
\section{Shape theory and boundaries of $\CAT(0)$ groups}
\label{sec:ShapeTheory}

This appendix contains the proof of Proposition~\ref{prop:LCZBoundary}, as well as a brief introduction to the ideas from shape theory involved in the proof.

\begin{defn}[Shape equivalence]
Let $Q = \prod_{k=1}^\infty [0,1/k]$ denote the Hilbert cube, which is a compact $\CAT(0)$ space when endowed with the $\ell^2$ metric.
Every compact metric space $A$ embeds as a $\mathcal{Z}$--set in the Hilbert cube $Q$.  Indeed $Q=[0,1]\times Q'$, where $Q'$ is homeomorphic to $Q$. The Urysohn metrization theorem gives an embedding of $A$ into the subspace $\{1\}\times Q'$ of $Q$ whose image is a $\mathcal{Z}$--set in $Q$.
All $\mathcal{Z}$--set embeddings of a given compactum $M$ in $Q$ have homeomorphic complements (see \cite{DydakSegal78}).

Chapman's Complement Theorem \cite{Chapman_Lectures}, states that shape equivalence of compact metric spaces may be characterized in the following manner, which we take as a definition:  Two metric compacta $M$ and $M'$ are \emph{shape equivalent} if they admit $\mathcal{Z}$--set embeddings into $Q$ such that the complements $Q \setminus M$ and $Q \setminus M'$ are homeomorphic.
A more geometric interpretation of shape equivalence is discussed in \cite[\S 17.7]{Geoghegan08}.
\end{defn}

For convenience, we repeat the statement of Proposition~\ref{prop:LCZBoundary}.

\begin{PropLCZBoundary}
Suppose $G$ is one ended and acts geometrically on a $\CAT(0)$ space $X$.
Then $G$ is semistable at infinity if and only if $G$ acts geometrically on a $\CAT(0)$ space $X'$ that admits a locally connected $\mathcal{Z}$--boundary.
\end{PropLCZBoundary}

The strategy used in the following proof was suggested to the authors by Craig Guilbault.

\begin{proof}
The reverse implication follows immediately from Theorem~\ref{thm:LocallyConnected}, so we focus on the forward implication.
As mentioned above, the visual boundary $\boundary X$ is a $\mathcal{Z}$--set in the compactification $Y = X \cup \boundary X$.
If $Y$ is any compact AR, then $Y \times Q$ is homeomorphic to $Q$.
Indeed a theorem of Edwards states that if $Y$ is any ANR then $Y \times Q$ is a Hilbert cube manifold \cite{Chapman_Lectures}, and a theorem of Chapman states that any compact contractible Hilbert cube manifold is homeomorphic to $Q$ \cite{Chapman_Lectures}.
Since $\boundary X$ is a $\mathcal{Z}$--set in $Y = X \cup \boundary X$, it follows that $\boundary X \times Q$ is a $\mathcal{Z}$--set in $Y \times Q = Q$.
Since they are homotopy equivalent, $\boundary X$ and $\boundary X \times Q$ are shape equivalent.
By Theorem~\ref{thm:LocallyConnected}, if $G$ is semistable at infinity then $\boundary X$ is shape equivalent to a locally connected continuum $A$.
If we choose any $Z$--set embedding $A \to Q$, shape equivalence implies that the complement $Q \setminus A$ is homeomorphic to $(Y\times Q) \setminus (\boundary X \times Q) = X \times Q$.
Therefore $A$ is a locally connected $\mathcal{Z}$--boundary of $X \times Q$.
To complete the proof of the forward implication, we extend the action of $G$ to the $\CAT(0)$ space $X \times Q$ by letting $G$ act trivially on the compact $Q$ factor.
\end{proof}

\bibliographystyle{alpha}
\bibliography{chruska.bib}

\end{document}